\documentclass{amsart}
\usepackage{amsfonts,amssymb,amsthm,cite,amsmath,amstext,amsbsy,bm}
\usepackage[toc,page]{appendix}
\usepackage{paralist}
\usepackage{color}
\usepackage[dvipsnames]{xcolor}
\usepackage[colorlinks,linkcolor=Bittersweet,citecolor=Cerulean]{hyperref}
\usepackage{mathrsfs}
\usepackage{euscript}
\usepackage[all]{xy}
\usepackage{dsfont}
\usepackage{mathabx}
\usepackage{mathtools}
\usepackage[text={170mm,260mm},centering]{geometry}

\makeatletter
\newsavebox{\@brx}
\newcommand{\llangle}[1][]{\savebox{\@brx}{\(\m@th{#1\langle}\)}%
  \mathopen{\copy\@brx\kern-0.5\wd\@brx\usebox{\@brx}}}
\newcommand{\rrangle}[1][]{\savebox{\@brx}{\(\m@th{#1\rangle}\)}%
  \mathclose{\copy\@brx\kern-0.5\wd\@brx\usebox{\@brx}}}
\makeatother

\DeclarePairedDelimiter\floor{\lfloor}{\rfloor}

\newtheorem{thmx}{Theorem}

\numberwithin{equation}{section}

\definecolor{plum}{rgb}{0.8,0.2,0.8}

\title{Applications of the Ohsawa-Takegoshi Extension Theorem to Direct Image Problems}

\date{\today}
\author{Ya Deng}
\address{Institut de Recherche Math\'ematique Avanc\'ee\\
	Universit\'e de Strasbourg}
\email{ deng@math.unistra.fr;\ dengya.math@gmail.com}

\begin{document}
\maketitle

\newtheorem{thm}{Theorem}[section]
\newenvironment{thmbis}[1]
{\renewcommand{\thethm}{\ref{#1}\emph{bis}}%
	\addtocounter{thm}{-1}%
	\begin{thm}}
	{\end{thm}}

\newtheorem{rem}[thm]{Remark}
\newtheorem{problem}[thm]{Problem}
\newtheorem{conjecture}[thm]{Conjecture}
\newtheorem{lem}[thm]{Lemma}
\newtheorem{cor}[thm]{Corollary}
\newtheorem{dfn}[thm]{Definition}
\newtheorem{proposition}[thm]{Proposition}
\newtheorem{example}[thm]{Example}
\newtheorem{propx}[thmx]{Proposition}
\newtheorem{corx}[thmx]{Corollary}

\def\oc{\mathscr{O}} \def\oce{\mathcal{O}_E} \def\xc{\mathcal{X}}
\def\ac{\mathcal{A}} \def\rc{\mathcal{R}} \def\mc{\mathcal{M}}
\def\wc{\mathcal{W}} \def\fc{\mathcal{F}} \def\cf{\mathcal{C_F^+}}
\def\jc{\mathscr{J}}
\def\ic{\mathcal{I}}
\def\kc{\mathcal{K}}
\def\lc{\mathcal{L}}
\def\vc{\mathcal{V}}
\def\xc{\mathcal{X}}
\def\yc{\mathcal{Y}}
\def\af{\mathbf{a}}
\def\df{\mathbf{d}}

\def\cb{\mathbb{C}}
\def\ib{\mathbb{I}}
\def\ab{\mathbb{A}}
\def\vol{\operatorname{vol}}
\def\ord{\operatorname{ord}}
\def\Im{\operatorname{Im}}
\def\dm{\mathrm{d}}

\def\as{{a^\star}} \def\es{e^\star} \def\cs{\mathscr{C}}

\def\tl{\widetilde} \def\tly{{Y}} \def\om{\omega}

\def\cb{\mathbb{C}} \def\nb{\mathbb{N}}
\def\gb{\mathbb{G}} \def\nbs{\mathbb{N}^\star}
\def\pb{\mathbb{P}} \def\pbe{\mathbb{P}(E)} \def\rb{\mathbb{R}}
\def\zbb{\mathbb{Z}}
\def\ys{\mathscr{Y}}
\def\ls{\mathscr{L}}
\def\hs{\mathscr{H}}
\def\xs{\mathscr{X}}
\def\zs{\mathscr{Z}}
\def\grs{{\rm{Gr}}_{k+1}(V)}
\def\ebf{\mathbf{e}}
\def\yf{\mathfrak{Y}}
\def\gf{\mathbf{G}}
\def\bff{\mathbf{b}}
\def\wk{\mathfrak{w}}
\def\ir{{\rm i}}

\def\hb{\bold{H}} \def\fb{\bold{F}} \def\eb{\bold{E}}
\def\pbb{\bold{P}}

\def\nab{\overline{\nabla}} \def\n{|\!|} \def\spec{\textrm{Spec}\,}
\def\cinf{\mathcal{C}_\infty} \def\d{\partial}
\def\db{\overline{\partial}}
\def\hess{\sqrt{-1}\partial\overline{\partial}}
\def\zb{\overline{z}} \def\lra{\longrightarrow}
\def\tb{\overline{t}}
\def\sn{\sqrt{-1}}
\begin{abstract}
In the first part of the paper,   we study a Fujita-type conjecture by Popa and Schnell, and give an effective bound on the generic global generation of the direct image of the twisted pluricanonical bundle. We also point out the relation between the Seshadri constant and the optimal bound. In the second part, we give  an affirmative answer to a question by Demailly-Peternell-Schneider in a more general setting. As an application, we generalize the theorems by Fujino and Gongyo on  images of weak Fano manifolds to the Kawamata log terminal cases, and refine a result by Broustet and Pacienza on the rational connectedness of the image.
\end{abstract}
\tableofcontents

\section{\sc Introduction}\label{introduction}
The first goal of this paper  is to study the following conjecture by Popa and Schnell on the global generation of pushforwards of pluricanonical bundles twisted by ample line bundles. 
\begin{conjecture}(Popa-Schnell)\label{PS}
	Let $f:X\rightarrow Y$ be a surjective morphism of smooth
	projective varieties,  with ${\rm dim}(Y)=n$, and let $L$ be an ample line bundle on $Y$. Then, for every
	$k\geq 1$, the sheaf 
	$$f_*(K_X^{\otimes k})\otimes L^l$$
	is globally generated for
	any $l\geq k(n+1)$.
\end{conjecture}

\medskip

In \cite{PS14}, Popa and Schnell proved the conjecture in the case when $L$ is an ample and globally generated line bundle, and in general when ${\rm dim}(X)=1$.  In a  recent preprint \cite{Dut17}, Dutta  was able to  remove the global generation assumption on
$L$ making a statement about generic global generation with weaker bound on the twist, as in the work of Angehrn and Siu \cite{AS95}, on the effective freeness of adjoint bundles. Her theorem is as follows:
\begin{thm}[Dutta]\label{Dut}
	Let $f:X\rightarrow Y$ be a surjective morphism of  
	projective varieties,  with $Y$ smooth and ${\rm dim}(Y)=n$. Let $L$ be an ample line bundle on $Y$.  Consider a klt  pair $(X,\Delta)$ on $X$, with $\Delta$ $\mathbb{Q}$-effective, such that $k(K_X+\Delta)$ is a  Cartier divisor for some $k\geq 1$. Denote $P=\oc_X\big(k(K_X+\Delta)\big)$. Then, for every
	$m\geq 1$, the sheaf 
	$$f_*P\otimes L^l$$
	is generated  by  global  sections  at  a  general  point $y\in Y$, either
	\begin{enumerate}[(a)]
		\item for all $l\geq k\big(\binom{n+1}{2}+1\big)$\\
		or
		\item for all $l\geq k(n+1)$ when $n\leq 4$.
	\end{enumerate}
\end{thm}
Here $\binom{n+1}{2}$ is the Angehrn-Siu type bound in their work on the Fujita conjecture \cite{AS95}. 

\medskip

Inspired by Demailly's recent work on the Ohsawa-Takegoshi type extension theorem \cite{Dem15} and P\u{a}un's proof of Siu's invariance of plurigenera \cite{Pau07}, we are able to prove the following theorem:
\begin{thmx}\label{mainot}
	Let $f:X\rightarrow Y$ be a morphism of smooth
	projective varieties,  with ${\rm dim}(Y)=n$, and let $L$ be an ample line bundle on $Y$.   If $y$ is a regular value of $f$, then for every
	$k\geq 1$, the sheaf 
	$$f_*(K_X^{\otimes k})\otimes L^l$$
	is generated  by  global  sections  at  $y$ for
	any $l\geq k(\floor*{ \frac{n}{\epsilon(L,y)}}+1)$. Here $\epsilon(L,y)>0$ is the Seshadri constant of $L$ at the point $y$.
\end{thmx}

Motivated in part by his study of linear
series in connection with the Fujita conjecture, Demailly  introduced the \emph{Seshadri constant} to measure the local positivity of the ample line bundle at a point  \cite{Dem92}.
After that Ein and Lazarsfeld systematically studied the Seshadri constant, and they first proved that for any ample line bundle $L$ on a projective surface $Y$, the Seshadri constant 
$$
\epsilon(L,y)\geq 1
$$
for a very general point on $Y$ \cite{EL93}. Inspired by this result, they further raised the following conjecture:
\begin{conjecture}[Ein-Lazarsfeld]\label{EL conj}
	Let $Y$ be any projective manifold, and $L$ any ample line bundle on $Y$. Then the Seshadri constant 
	$$\epsilon(L,y)\geq 1$$ 
	at a very general point $y\in Y$.
\end{conjecture} 
In \cite{EKL95}, they proved the existence of  universal generic bound in a fixed dimension. However, the bound is suboptimal by a factor of $n={\rm dim}(Y)$. 
\begin{thm}[Ein-K\"uchle-Lazarsfeld]\label{EKLmain}
	Let $Y$ be a projective variety, and $L$ an ample line bundle on $Y$. Then for any given $\delta>0$, the locus 
	$$
	\{y\in Y| \epsilon(L,y)>\frac{1}{n+\delta} \}
	$$
	contains a Zariski-dense open set in $Y$.
\end{thm}
Applying Theorem \ref{EKLmain} to our Theorem \ref{mainot}, we have the following general result:
\begin{propx}\label{sesha main}
With the same setting in Theorem \ref{mainot}.  For any $k\geq 1$, the direct image
	$$f_*(K_X^{\otimes k})\otimes L^l$$
	is generated  by  global  sections  at  general points of $Y$ for
	any $l\geq k(n^2+1)$. In particular, if the manifold $Y$ satisfies Conjecture \ref{EL conj}, then Conjecture \ref{PS} holds true for general points in $Y$; that is, the direct image
	$$f_*(K_X^{\otimes k})\otimes L^l$$
	is generated  by  global  sections  at   general points of $Y$ for
	any $l\geq k(n+1)$.
\end{propx}

Compared to Theorem \ref{Dut} by Dutta, our bound for $l$ is also quadratic on $n$ but slightly weaker than hers. However, if we apply the  result that $K_Y+(n+1)L$ is semi-ample for any ample line bundle $L$, we can obtain a linear bound for $l$. Moreover, when $K_Y$ is pseudoeffective, $l$ can be taken independent of $n$.

\begin{thmx}\label{sharp bound}
	Let $f:X\rightarrow Y$ be a surjective morphism of  
projective varieties,  with $X$  normal, $Y$  smooth  and ${\rm dim}(Y)=n$. Let $L$ be an ample line bundle over $Y$.  Consider a klt $\mathbb{Q}$-pair $(X,\Delta)$ on $X$, with $\Delta$ effective. Then for any positive integer $m$ such that $m(K_X+\Delta)$ is a (integral) Cartier divisor,    the direct image
$$f_*\big(m(K_X+\Delta)\big)\otimes L^l$$
is generated  by  global  sections  at  a  general  point $y\in Y$, either
\begin{enumerate}[(a)]
	\item for all $l\geq m(n+1)+n^2-n$\\
	\item for all $l\geq n^2+2$ when $K_Y$ is pseudo-effective.
\end{enumerate}
\end{thmx}

\medskip

The second part of the paper is to study a question   by Demailly-Peternell-Schneider in \cite{DPS01}:
\begin{problem}\label{DPS}
	Let $X$ and $Y$ be normal projective $\mathbb{Q}$-Gorenstein varieties. Let $f:X\rightarrow Y$
	be a surjective morphism. If $-K_X$ is pseudo-effective and  its non-nef locus does not project onto $Y$, is $-K_Y$ pseudo-effective?
\end{problem}

Inspired by the recent work of J. Cao on the local isotriviality on the Albanese map of  projective manifolds with nef anticanonical bundles \cite{Cao16}, we give an affirmative answer to the above problem in a more general setting:
\begin{thmx}\label{DPS01}
		Let $X$ be a normal projective variety and $D$ an
	effective $\mathbb{Q}$-divisor on $X$ such that the pair $(X,D)$ is log canonical. Let $Y$ be a normal projective $\mathbb{Q}$-Gorenstein variety, and $f:X\rightarrow Y$ is a surjective morphism.
	\begin{enumerate}[(a)]
		\item \label{claim1} Assume that $-(K_X+D)$ is pseudo-effective, such that the non-nef locus ${\rm NNef}\big(-(K_X+D)\big)$  does not project onto $Y$ via $f$. Then $-K_Y$ is pseudo-effective. 
		\item \label{claim2} If we further assume that both $X$ and $Y$ are smooth, with $\Delta$ a (not necessarily effective) $\mathbb{Q}$-divisor on $Y$, such that $-(K_X+D)-f^*\Delta$ is pseudo-effective, with its non-nef locus does not map onto $Y$. 	Then $-K_Y-\Delta$ is pseudo-effective with its non-nef locus contained in $f\big({\rm NNef}(-K_X-D-f^*\Delta)\big) \bigcup Z\bigcup Z_D $, where $Z$ is the minimal proper subvariety on $Y$ such that $f:X\setminus f^{-1}(Z)\rightarrow Y\setminus Z$ is smooth, and $Z_D$ is an at most countable union of proper subvarities containing $Z$ such that for every $y\notin Z_D$, the pair $\big(f^{-1}(y), D_{\upharpoonright f^{-1}(y)}\big)$ is also lc.
	\end{enumerate} 
\end{thmx}

The following theorem by Fujino and Gongyo \cite[Theorem 1.1]{FG14} is a direct consequence of our Claim \eqref{claim2} in Theorem \ref{DPS01}.
\begin{thm}[Fujino-Gongyo]\label{FG2}
	Let $f:X\rightarrow Y$ be a smooth fibration between smooth projective varieties. Let $D$ be
	an effective $\mathbb{Q}$-divisor on $X$ such that $(X, D)$ is lc, ${\rm Supp}(D)$ is a simple normal crossing divisor, and ${\rm Supp}(D)$ is relatively normal crossing over $Y$. Let $\Delta$ be a (not necessarily effective) $\mathbb{Q}$-divisor on $Y$. Assume that $-(K_X+D)-f^*\Delta$ is nef. Then so is $-K_Y-\Delta$.
\end{thm}

Moreover, we can also use analytic methods to prove the following theorem.
\begin{thmx}\label{main2}
		Let $X$ be a normal projective variety and $D$ an
	effective $\mathbb{Q}$-divisor on $X$ such that $K_X+D$ is $\mathbb{Q}$-Cartier. Let $Y$ be a normal projective $\mathbb{Q}$-Gorenstein variety, and $f:X\rightarrow Y$ is a surjective morphism.
	\begin{enumerate}[(a)]
		\item\label{claima} Assume that both $X$ and $Y$ are smooth, with $(X,D)$   klt. Let $\Delta$ be a (not necessarily effective) $\mathbb{Q}$-divisor over $Y$, such that  $-K_X-D-f^*\Delta$ is big and its non-nef locus ${\rm NNef}(-K_X-D-f^*\Delta)$ does not dominate $Y$, then $-K_Y-\Delta$ is big.
		\item\label{claimb} Assume that the restriction of $f$ to ${\rm NNef}(-K_X-D)\bigcup {\rm Nklt}(X,D)$  does not dominate $Y$, then $-K_Y$ is  big. Here ${\rm Nklt}(X,D)$ denotes the \emph{non-klt locus} of $(X,D)$.
	\end{enumerate}
\end{thmx}
As a combination of Theorem \ref{DPS01} and \ref{main2}, we prove the following result, which is a generalization of another theorem by Fujino and Gongyo \cite[Theorem 1.1]{FG12}.
\begin{corx}\label{main3}
Let $f:X\rightarrow Y$ be a smooth fibration between smooth projective manifolds. Assume that $\Delta$ is  a (not necessarily effective) $\mathbb{Q}$-divisor over $Y$, $(X,D)$ is klt,  and $(X_y,D_{\upharpoonright X_y})$ is lc for every $y\in Y$.     If $-K_X-D-f^*\Delta$ is big and nef, so is $-K_Y-\Delta$.
\end{corx}

Finally, we apply Claim \eqref{claimb} in Theorem \ref{main2} to refine a result by Broustet and Pacienza on the rational connectedness of the image (compared to Theorem \ref{BP} below):
\begin{thmx}\label{RC}
	Let $X$ be a normal projective variety and $D$ an
	effective $\mathbb{Q}$-divisor on $X$ such that $K_X+D$ is $\mathbb{Q}$-Cartier. Let
	$Y$ be a normal and $\mathbb{Q}$-Gorenstein projective variety. If $f : X \rightarrow Y$ is
	a surjective morphism such that $-(K_X+D)$ is big and the restriction of $f$ to ${\rm NNef}(-K_X-D)\bigcup {\rm Nklt}(X,D)$  does not dominate $Y$,  then $Y$ is rational connected  modulo  
	${\rm NNef}(-K_Y)$, that is, there
	exists an irreducible 
	component $V$ of ${\rm NNef}(-K_Y)$ such that for any general point $y$ of $Y$ there exists a rational 
	curve $R_y$ passing through $y$ and intersecting
	$V$. 
\end{thmx}

\section{\sc Technical Preliminaries}\label{preliminary}
\subsection{Definitions and Notations}
\begin{dfn}
	A pair consists of a normal variety $X$, together with a Weil $\mathbb{Q}$-divisor $\Delta=\sum d_i\Delta_i$ on $X$, such that the $\mathbb{Q}$-divisor  $K_X+\Delta$ is $\mathbb{Q}$-Cartier on $X$.
\end{dfn}
For a pair $(X,\Delta)$ with $\Delta$ effective, the non-klt locus of  is defined by 
$$
{\rm Nklt}(X,\Delta):=\{x\in X| \jc(X,\Delta)_x\neq \oc_{X,x} \}.
$$
Let us recall the definitions of non-nef locus and restrict base locus for any pseudo-effective line bundle on the normal projective variety in \cite[Definition 1.5]{BBP13}:
\begin{dfn}\label{restrict} Let $X$ be a normal projective variety. Let $D$ be a pseudo-effective $\mathbb{R}$-divisor on $X$. The non-nef locus of $D$ is
	defined as
	$${\rm NNef}(D):= \{	c_X(v) | v(\lVert D\rVert)> 0\},$$
	where $c_X(v)$ denotes the center on $X$ of a given divisorial valuation $v$. The restricted base locus of $D$ is defined by 
	$$
	\mathbf{B}_-(D):=\bigcup_{m>0} \mathbf{B}(D+\frac{1}{m}A),
	$$
	where $A$ is an ample divisor, and $\mathbf{B}(\bullet)$ denotes \emph{stable base locus} of the $\mathbb{R}$-divisor.
\end{dfn}
 It was proved in \cite[Lemma 1.6]{BBP13} that, 
$$
{\rm NNef}(D)\subset 	\mathbf{B}_-(D),
$$
and equality was shown to hold when $X$ is smooth.
\subsection{Seshadri Constants}

In the work \cite{Dem92}, Demailly define the following \emph{Seshadri constant}:
\begin{dfn}
	Le $L$
	be a nef line bundle over a projective algebraic manifold $X$. To every point
	$x\in X$, one defines the number
	$$
	\epsilon(L,x):=\inf \frac{L\cdot C}{\nu(C,x)}
	$$
	where the infimum is taken over all reduced irreducible curves $C$ passing through $x$
	and $\nu(C,x)$ is the multiplicity of $C$ at $x$.  $\epsilon(L,x)$ will be called the \emph{Seshadri constant} $L$ at $x$.
\end{dfn}

\medskip

On the other hand, Demailly also introduced another constant $\gamma(L,x)$ for any nef line bundle $L$. First, we begin with the following definition.
\begin{dfn}
	A function $\psi:X\rightarrow ]-\infty,+\infty]$ on a complex manifold
	$X$ of dimension $m$ is said to
	be
	quasi-plurisubharmonic  (quasi-psh for short) if $\psi$ is  locally  the  sum  of  a  psh  function  and  of
	a  smooth  function
	(or  equivalently,  if $\hess \psi$ is  locally  bounded  from  below)
	.  In  addition,
	we  say  that $\psi$ has
	neat analytic singularities
	if  every  point $x\in X$
	possesses  an  open
	neighborhood $U$ on which $\psi$
	can be written
	$$\psi=c\log \sum_{j=1}^{N}|g_j|^2+w(z) $$
	where $g_j\in \oc(U)$, $c\geq 0$ and $w(z)\in \cs^{\infty}(U)$.
\end{dfn}

\begin{dfn}
	A singular metric $h$ on the line bundle $L$   
	is said to have a \emph{logarithmic pole of coefficient $\nu$} at a point $x\in X$,
	if  on a neighborhood $U$ of $x$,  the local weight $\varphi$ of $h$ can be written
	$$
	\varphi=\nu \log \sum |z-x|^2+ w(z)
	$$
	where $\nu> 0$ and $w(z)\in \cs^{\infty}(U)$. In this setting, we set $\nu(h,x):=\nu$.
\end{dfn} 
Then we set
$$\gamma(L,x):=\sup_h \nu(h,x),$$ where the supremum  is taken over all singular hermitian metrics $h$ of $L$ with positive curvature current, whose local weight $\varphi$ has neat singularities and logarithmic poles at $x$. 

The numbers $\epsilon(L,x)$ and $\gamma(L,x)$ will be seen to carry a lot of useful information about the local positivity of $L$. In case $L$ is big and nef,   these two constants coincide outside a certain proper subvariety of $X$ (see \cite[Theorem 6.4]{Dem92})
\begin{thm}[Demailly]\label{Demailly constant}
	Let $L$ be a big and nef line bundle over $X$. Then we have
	$$
	\epsilon(L,x)=\gamma(L,x)
	$$
	for any $x\notin \mathbf{B}_+(L)$, where $\mathbf{B}_+(L)$ is the  augmented base locus of $L$ (see \cite[Definition 10.2.2]{Laz04}). In particular, if $L$ is ample, then $\epsilon(L,x)=\gamma(L,x)$ holds everywhere.
\end{thm}

\medskip

As we mentioned in Section \ref{introduction}, in \cite{EKL95}, Ein, K\"uchle and Lazarsfeld gave the existence of universal generic bounds for the Seshadri constants in a fixed dimension.
\begin{thm}[Ein-K\"uchle-Lazarsfeld]\label{EKL}
	Let $Y$ be an irreducible projective variety of dimension $n$, and $L$ a nef line bundle on $Y$.  Suppose there exists a countable union $\mathcal{B}\subset Y$ of proper subvarieties of
	$Y$ 	plus a positive real number $\alpha>0$ 	such that
	\begin{eqnarray}
	L^r\cdot Z\geq (\alpha\cdot r)^r
	\end{eqnarray}
	for every irreducible subvariety $Z\subset Y$ of dimension $r$ ($1\leq r\leq n$) with $Z\not\subseteq \mathcal{B}$.
	Then $$\epsilon(L,y)\geq\alpha$$
	for all $y\in Y$ outside a countable union of proper subvarieties in $Y$.
	In particular, for any ample line bundle $L$ on $Y$,
	\begin{eqnarray}\label{lower}
	\epsilon(L,y)\geq \frac{1}{n}
	\end{eqnarray}
	for a very general point $y$.
\end{thm}

The above theorem gives a lower bound on the Seshadri constant of
a nef and big line bundle at a very general point. However, as was also proved in \cite{EKL95}, for the ample line bundle, the above theorem is valid on a Zariski-open set by the semi-continuity of the Seshadri constant of the ample line bundle. In other word, let
$L$ be an ample line bundle on an irreducible projective variety $Y$. Suppose that there is a positive rational number $B$
and a smooth point $y\in Y$
for
which one knows that
$$
\epsilon(L,y)>B.
$$
Then the locus
$$
\{z\in Y| \epsilon(L,z)>B\}
$$
contains a Zariski-open dense set in $Y$.

\subsection{$L^2$ Extension Theorem}

Before we state Demailly's Ohsawa-Takegoshi type Extension Theorem,  we begin with a definition in \cite{Dem15}.

\begin{dfn}
	If $\psi$ is a quasi-psh function on a complex manifold $X$, the multiplier
	ideal sheaf $\jc(\psi)$ is the coherent analytic subsheaf of $\oc_X$ defined by
	$$\jc(\psi)_x:=\{f\in \oc_{X,x}; \exists U\ni x, \int_U|f|^2e^{-\psi}\dm \lambda<+\infty \}$$
	where
	$U$ is an open coordinate neighborhood of $x$, and $\dm \lambda$  the standard Lebesgue measure
	in the corresponding open chart of $\cb^n$.  We say that the singularities of $\psi$
	are log canonical
	along the zero variety
	$Y:=V(\jc(\psi))$ if $\jc\big((1-\epsilon)\psi\big)_{\upharpoonright Y}=\oc_{X\upharpoonright Y}$ for every $\epsilon>0$.
\end{dfn}

If $\psi$ possesses both neat and log canonical singularities, it is easy to show that the zero scheme $V(\jc(Y))$ is a reduced variety. In this case one  can  also  associate  in  a
natural way a measure $dV_{Y^\circ,\omega}[\psi]$ on the set $Y^\circ:=Y^{\rm reg}$ of regular points of $Y$ as follows. If
$g\in \cs_c(Y^\circ)$ is a compactly supported continuous function on $Y^\circ$,
and $\tilde{g}$ compactly supported extension of $g$ to $X$, we set
\begin{eqnarray}\label{measure}
\int_{Y^\circ} g dV_{Y^\circ,\omega}[\psi]:=\limsup_{t\rightarrow -\infty} \int_{x\in X, t<\psi(x)<t+1} \tilde{g}(x) dV_{X,\omega}.
\end{eqnarray}
Here $\omega$ is a K\"ahler metric on $X$, and $dV_{X,\omega}=\frac{\omega^m}{m!}$. In \cite{Dem15} Demailly proved that the limit does not depend on the continuous extension $\tilde{g}$, and one gets in this way a measure with smooth positive density with respect to the Lebesgue measure, at least on an (analytic) Zariski open set in $Y^\circ$. 

\medskip

We are ready to recall the Ohsawa-Takegoshi type extension Theorem by Demailly. We only need a special case of his very general statement:
\begin{thm}[Demailly]\label{OT}
	Let $X$
	be  a smooth projective  manifold,  and
	$\omega$ a  K\"ahler  metric  on $X$.   Let $L$
	be
	a  holomorphic  line  bundle  equipped  with  a  (singular)  hermitian  metric
	$h$ on X,  and  let
	$\psi:X\rightarrow ]-\infty,+\infty]$
	be  a  quasi-psh  function  on $X$ with  neat  analytic  singularities.  Let
	$Y$
	be  the  analytic  subvariety  of
	$X$ defined  by $Y=V(\jc(Y))$
	and  assume  that
	$\psi$
	has  log canonical  singularities  along
	$Y$,  so  that $Y$ is reduced. Finally, assume that the curvature current
	$$
	{\rm i}\Theta_{L,h}+\alpha\hess \psi\geq 0
	$$
	for all $\alpha\in [1,1+\delta]$ and some $\delta>0$. Then for every section $s\in H^0\big(Y^\circ,(K_X\otimes L)_{\upharpoonright Y^\circ}\big)$ on $Y^\circ:=Y^{\rm reg}$ such that
	$$
	\int_{Y^\circ} |s|^2_{\omega,h}dV_{Y^\circ,\omega}[\psi]<+\infty,
	$$
	there is an extension of $S\in H^0(X,K_X\otimes L)$ whose restriction to $Y^\circ$ is equal to $s$, such that
	$$
	\int_{X} \gamma(\delta \psi)|S|^2_{\omega,h}e^{-\psi}dV_{X,\omega}\leq\frac{34}{\delta}\int_{Y^\circ} |s|^2_{\omega,h}dV_{Y^\circ,\omega}[\psi].
	$$
\end{thm}
Here we set 
\begin{equation}
\gamma=\left\{
\begin{aligned}
&e^{-\frac{x}{2}} &\text{if} \quad x\geq 0,\\
&\frac{1}{1+x^2} &\text{if} \quad x<0.
\end{aligned}
\right.\nonumber
\end{equation}
\medskip

A direct consequence of	Theorem \ref{OT} is the following extension theorem for fibrations:
\begin{cor}\label{fibration extension}
	Let $f:X\rightarrow Y$ be a surjective morphism between smooth manifolds. For any ample line bundle $L$ on $Y$, any regular value $y$ of $f$, if the Seshadri constant of $L$ satisfies that
	\begin{eqnarray}\label{Seshadri condition}
	\epsilon(L,y)>{\rm dim}(Y)=n,
	\end{eqnarray}
	then for any pseudo-effective line bundle $L_1$ over $X$ with a singular hermitian metric $h$ such that $\Theta_{L_1,h}\geq0$, and the restriction of $h$ to $X_y$ is not identically zero, any section $s$ of
	$$
	H^0\big(X_y,(K_X\otimes f^*L\otimes L_1)_{\upharpoonright X_y}\otimes \jc(h_{\upharpoonright X_y})\big).
	$$
	can always be extended to a global one 
	$$
	S\in H^0\big(X, K_X\otimes f^*L\otimes L_1\big)
	$$
	with certain $L^2$ estimates which do not depend on $L_1$.
\end{cor}
\begin{proof}
	Since $L$ is ample over $Y$, one can find a smooth hermitian metric $h_0$ on $L$  with the curvature form ${\rm i}\Theta_{L,h_0}\geq \omega$, where $\omega$ is some K\"ahler form on $Y$.
	
	By the lower bound of Seshadri constant $\epsilon(L,y)>n$,  we can find a global quasi-psh function $\varphi$ with neat singularities on $Y$ such that
	\begin{enumerate}[(a)]
		\item $
		\ir\Theta_{L,h_0}+\hess \varphi\geq 0;
		$
		\item $\varphi$ is smooth outside $y$;
		\item on a neighborhood $W$ of $y$, we have
		$$
		\varphi=(1+\delta)n\log \sum |z-y|^2+w(z)
		$$
		where $\delta>0$ and $w(z)\in \cs^{\infty}(W)$ with $w(y)=0$
	\end{enumerate}
	Now set $\psi:=\frac{1}{1+\delta}\varphi\circ f$, which is a quasi-psh function with neat singularities on $X$. 
	Moreover, since $y$ is the regular value of $f$,  the inverse image $X_y:=f^{-1}\{y\}$ is a finite disjoint union of closed smooth submanifolds of codimension $n$ in $X$, and the multiplier ideal sheaf
	$$\jc(\psi)=\jc(\ic^{\langle n\rangle}_{X_y})=\ic_{X_y}.$$ 
	Here $\ic^{\langle n\rangle}_{X_y}$ is the ideal sheaf consisting of germs of functions that have multiplicity $\geq n$ at a general point of $X_y$:
	$$\ic^{\langle n\rangle}_{X_y}:=\{f\in \oc_X | \ord_x(f)\geq n \mbox{ for a general point }x\in X\}.$$
	Thus $\jc(\psi)$ has log canonical singularities, and we have
	$$
	{\rm i}\Theta_{L_1,h}+{\rm i}\Theta_{f^*L,f^*h_0}+\alpha\hess \psi\geq 0
	$$
	for all $\alpha\in [1,1+\delta]$. Then for any  section $s$ of
	$$
	H^0\big(X_y,(K_X\otimes f^*L\otimes L_1)_{\upharpoonright X_y}\otimes \jc(h_{\upharpoonright X_y})\big),
	$$
	we can apply Theorem \ref{OT}  to extend $s$ to a global section
	$$
	S\in H^0\big(X, K_X\otimes f^*L\otimes L_1\otimes \jc(h)\big)
	$$
	such that
	$$
	\int_{X} \gamma(\delta \psi)|S|^2_{\omega,f^*h_0h_1}e^{-\psi}dV_{X,\omega}\leq\frac{34}{\delta}\int_{X_y} |s|^2_{\omega,f^*h_0h_1}dV_{X_y,\omega}[\psi].
	$$
	Assume that ${\rm dim}(X)=m+n$. From (\ref{measure})	one can then  check that $dV_{X_y,\omega}[\psi]$ is the smooth measure supported on $X_y$, such that 
	$$
	dV_{X_y,\omega}[\psi]=C_0\frac{\omega_{\upharpoonright X_y}^{m}}{m!},
	$$
	where $C_0$ is some contant depending only on $m,n$. Since $\delta$ depends only on $\epsilon(L,y)$, write $C:=\frac{34}{\delta}C_0$ which does not depend on $L_1$. We thus obtain
	\begin{eqnarray}\label{estimate}
	\int_{X} \gamma(\delta \psi)|S|^2_{\omega,f^*h_0h_1}e^{-\psi}dV_{X,\omega}\leq C_0\int_{X_y} |s|^2_{\omega,f^*h_0h_1}\frac{\omega_{\upharpoonright X_y}^{m}}{m!},
	\end{eqnarray}
	where the $L^2$ estimate does not depend on $L_1$.
\end{proof}

\subsection{The Extension Theorem for Twisted Pluricanonical Bundles}\label{pluri}
We recall the following twisted pluricanonical extension theorem, which was inspired by that used by J. Cao to prove the local triviality of Albanese maps of projective manifolds with nef anticanonical bundles \cite{Cao16}. It is a 
consequence of \cite[Section A.2]{BP10}.
\begin{thm}\label{extension2}
	Let $Y$ be a $n$-dimensional projective manifold 
	and let $A_Y$ be any line bundle on $Y$ such that the difference $A_Y-K_Y$ is an ample line bundle. 	Let $f: X \rightarrow Y$ be a surjective morphism
	from a smooth projective manifold $X$ to $Y$ and   $L$ be a pseudo-effective line bundle on $X$ with a possible singular metric $h_L$ such that
	$$\ir \Theta_{h_L} (L) \geq 0 .$$ 
	Assume that there exists some regular value $z$ of $f$, we have
	\begin{enumerate}[(i)]
		\item \label{bpcon1} all the sections of the bundle $mK_{X_z}+L$ extend near $z$,
		\item \label{bpcon2}	$
		H^0 \big(X_z ,  (m K_{X_z} + L_{\upharpoonright X_z})\otimes\jc(h_{L\upharpoonright X_z}^{\frac{1}{m}}) \big)\neq \emptyset.
		$
	\end{enumerate}
	Then for any $y \in Y$  such that
	\begin{enumerate}[(a)]
		\item \label{con1} $y$ is the regular value of $f$,
		\item \label{con2} the Seshadri constant  $\epsilon(A_Y-K_Y,y)>n,$
		\item  \label{con3} all the sections of the bundle $(mK_X+L)_{\upharpoonright X_y}$ extend
		locally near $y$,
	\end{enumerate} 
	the restriction map
	$$H^0 (X, m K_{X/Y} + L  +f^* A_Y) \rightarrow H^0 \big(X_y ,  (m K_{X_y} + L_{\upharpoonright X_y})\otimes \jc(h_{L\upharpoonright X_y}^{\frac{1}{m}})\big)$$
	is surjective. 
	In particular, the choice of $A_Y$ depends only on $Y$ and is independent of $f, L, m$.
\end{thm}

\begin{proof}
	Thanks to \cite[A.2.1]{BP10}, the conditions \eqref{bpcon1} and \eqref{bpcon2} imply that	there exists a $m$-relative Bergman type metric $h_{m,B}$ on $m K_{X/Y} + L$ with respect to $h_L$ such that $i\Theta_{h_{m, B}} (m K_{X/Y}+L) \geq 0$. 
	Thus $h:= h_{m,B}^{\frac{m-1}{m}}\cdot h_L^{\frac{1}{m}}$ defines a possible singular metric on
	$$\widetilde{L} : = \frac{m-1}{m} (m K_{X/Y} + L) + \frac{1}{m} L=(m-1)K_{X/Y}+L ,$$
	with $i\Theta_h (\widetilde{L} ) \geq 0$.  
	
	Take any $s \in H^0 \big(X_y ,  (m K_{X_y} + L_{\upharpoonright X_y})\otimes \jc(h_{L\upharpoonright X_y}^{\frac{1}{m}})\big)$. It follows from Condition \eqref{con1}, \eqref{con2} and \eqref{con3}, and the construction of the $m$-relative Bergman kernel metric that 
	$|s|^2_{h_{m,B}}$ is $\cs^0$-bounded. 
	Then we see that 
	\begin{eqnarray*}\label{normcon}
		\int_{X_y} |s|_{\omega,h} ^2 dV_{X_y,\omega}&=&	\int_{X_y} |s|_{h_{m,B}} ^{\frac{2(m-1)}{m}}|s|^{\frac{2}{m}}_{\omega,h_L^{\frac{1}{m}}} dV_{X_y,\omega}\\
		&\leq& C \int_{X_y} |s|^{\frac{2}{m}}_{\omega,h_L^{\frac{1}{m}}} dV_{X_y,\omega}< +\infty. 
	\end{eqnarray*}
	We then can apply Corollary \ref{fibration extension} to $K_X + \widetilde{L} +f^* (A_Y - K_Y)$, 
	to extend $s$ to a section in $H^0 (X, K_{X/Y} + \widetilde{L} +f^* A_Y)$.
	In conclusion, the restriction
	$$H^0 (X, m K_{X/Y} + L +f^* A_Y) \rightarrow  H^0 \big(X_y ,  (m K_{X_y} + L_{\upharpoonright X_y})\otimes \jc(h_{L\upharpoonright X_y}^{\frac{1}{m}})\big)$$
	is  surjective and the theorem is proved. 
\end{proof}
One can see that there exists a non-empty Zariski open set of $Y^\circ\subset Y$ such that for any $y\in Y^\circ$, it satisfies  Condition \eqref{con1}, \eqref{con2} and \eqref{con3} in Theorem \ref{extension2}.

\section{\sc On the Conjecture of Popa and Schnell}
Let $f:X\rightarrow Y$ be the surjective morphism between smooth projective manifolds, and let $L$ be an ample line bundle on $Y$ with a smooth hermitian metric $h_0$ such that the curvature form $\ir \Theta_{h_0}\geq \omega$ for some K\"ahler metric $\omega$ on $Y$. Assume that ${\rm dim}(Y)=n$ and ${\rm dim}(X)=m+n$. Fix any point $y$ on $Y$ which is the regular value of $f$. Take any positive real number $\nu$ such that  $$\epsilon(L,y)>\frac{1}{\nu}.$$
Then we have
$$
\epsilon(\lceil n\nu\rceil L,y)>n.
$$
Set $\tilde{L}:=\lceil n\nu\rceil f^*L$ with the smooth hermitian metric $\tilde{h}:=f^*h_0^{\lceil n\nu\rceil}$, then we can restate Corollary \ref{fibration extension} in the following variant form:
\begin{proposition}\label{extension}
	There is a globally defined quasi-psh function $\psi_0$ defined over $X$ and a positive number $\delta$ such that,  for any pseudo-effective line bundle $L_1$ equipped with the  possible singular hermitian metric $h_1$, whose curvature current $\ir\Theta_{L_1,h_1}\geq 0$ and $h_1$ is not identically zero when restricted on $X_y$,  for any section 
	$$
	s\in H^0\big(X_y, (K_X\otimes \tilde{L}\otimes L_1)_{\upharpoonright X_y}\otimes \jc(h_{1 \upharpoonright X_y})\big),
	$$
	there is a global section
	$$
	S\in H^0(X, K_X\otimes \tilde{L}\otimes L_1)
	$$
	whose restriction to $X_y$ is $s$, such that
	$$
	\int_{X} \gamma(\delta \psi_0)|S|^2_{\omega,\tilde{h}h_1}e^{-\psi_0}dV_{X,\omega}\leq C \int_{X_y} |s|^2_{\omega,\tilde{h}h_1}dV_{X_y,\omega}.
	$$
	Here $dV_{X_y,\omega}:=\frac{\omega_{\upharpoonright X_y}^{m}}{m!}$, and $C$ is some constant which does not depend on $L_1$.
\end{proposition}

Thus from Proposition \ref{extension}, if we set $L_1$ to be the trivial bundle on $X$, we see that   the following morphism 
$$
H^0(X, K_X\otimes f^*L^{\otimes \lceil n\nu\rceil })\rightarrow H^0\big(X_y,(K_X\otimes f^*L^{\otimes \lceil n\nu\rceil})_{\upharpoonright X_y}\big)
$$ 
is always surjective. As one can take $\nu$ to be arbitrary close to $\frac{1}{\epsilon(L,y)}$ so that $\lceil n\nu\rceil=\floor*{\frac{n}{\epsilon(L,y)}}+1$, we see that the direct image $f_*K_X\otimes L^{\otimes \floor*{\frac{n}{\epsilon(L,y)}}+1}$ is generated by global sections at $y$. Since $y$ is an arbitrary regular value of $f$, we thus prove Theorem \ref{mainot} for $k=1$. In order to prove the theorem for any $k\geq 2$, we need to apply the techniques in proving Siu's invariance of plurigenera \cite{Siu97} by P\u{a}un \cite{Pau07}.
\medskip

\begin{proof}[Proof of Theorem \ref{mainot}]
	Fix any $k\geq 2$ and any $\sigma\in H^0\big(X_y,k(K_X+\tilde{L})_{\upharpoonright X_y}\big)$.	We want to find a global section $\Sigma\in H^0\big(X,k(K_X+\tilde{L})\big)$ whose restriction to $X_y$ is $\sigma$.
	
	\medskip
	
	Choose a very ample line bundle $A$ on $X$ such that for every $r=0,\ldots,k-1$, the line bundle $F_{0,r}:=r(K_X+\tilde{L})+A$ is globally generated by  sections
	$$\{u_j^{(0,r)}\}_{j=1,\ldots,N_r}\subset H^0(X,F_{0,r}).$$
	We then define inductively a sequence of line bundles $$F_{q,r}:=(qk+r)(K_X+\tilde{L})+A$$
	for any $q\geq 0$, and $0\leq r\leq k-1$. By constructions we have
	\begin{equation}\label{ad hoc}
	\left\{
	\begin{aligned}
	&F_{q,r+1}=K_X+F_{q,r}+\tilde{L} \quad &\text{if} \quad r<k-1,\\
	&F_{q+1,0}=K_X+F_{q,k-1}+\tilde{L} \quad &\text{if} \quad r=k-1.
	\end{aligned}
	\right.
	\end{equation}
	
	We are going to construct inductively families of sections, say $\{u_j^{(q,r)}\}_{j=1,\ldots,N_r}$, of $F_{q,r}$ over $X$, together with ad hoc $L^2$ estimates, such that each $u_j^{(q,r)}$ is an extension of $v_j^{(q,r)}$, where we set $$v_j^{(q,r)}:=\sigma^qu^{(0,r)}_{j \upharpoonright X_y}\in H^0(X,F_{q,r}).$$
	\medskip
	
	Now, by induction, assume that such $\{u_j^{(q,r)}\}_{j=1,\ldots,N_r}$ above can be constructed. Then $F_{p,r}$ can be equipped with a natural singular hermitian metric $h_{q,r}$ defined by
	$$
	|\xi|_{h_{q,r}}^2:=\frac{|\xi|^2}{\sum_{j=1}^{N_r}|u_j^{(q,r)}|^2},
	$$
	such that $\ir \Theta_{h_{q,r}}\geq 0$.  Let $h_{K_X}$ be the smooth hermitian metric of the canonical bundle $K_X$ induced by the volume form $dV_{X,\omega}$, and set $\hat{h}:=h_{K_X}\tilde{h}$ to be the smooth metric on $K_X+\tilde{L}$, then by construction the pointwise norm with respect to the metric $h_{q,r}$ is
	\begin{equation}\label{norm}
	\left\{
	\begin{aligned}
	&|v_j^{(q,r+1)}|^2_{\omega,h_{q,r}\tilde{h}}=\frac{|v_j^{(0,r+1)}|_{\hat{h}^{r+1}h_A}^2}{\sum_{i=1}^{N_r}|v_i^{(0,r)}|_{\hat{h}^rh_A}^2} \quad &\text{if} \quad r<k-1,\\
	&|v_j^{(q+1,0)}|^2_{\omega,h_{q,r}\tilde{h}}=\frac{|\sigma|_{\hat{h}^k}|v_j^{(0,0)}|_{h_A}^2}{\sum_{j=1}^{N_r}|v_i^{(0,r)}|_{\hat{h}^{k-1}h_A}^2} \quad &\text{if} \quad r=k-1.
	\end{aligned}
	\right.
	\end{equation}
	where $h_A$ is a smooth hermitian metric on $A$ with strictly positive curvature. Since the sections $\{v_i^{(0,r)}\}_{i=1,N_r}$  generates $F_{0,r  \upharpoonright X_y}$, there is a constant $C_1>0$ such that \eqref{norm} is uniformly $\cs^0$ bounded above by $C_1$. From \eqref{ad hoc}, it then follows from Proposition \ref{extension} that one can extend $v_j^{(q,r+1)}$ (or $v_j^{(q+1,0)}$ if $r=k-1$) into a section $u_j^{(q,r+1)}$ ($u_j^{(q+1,0)}$ respectively) over $X$  such that
	\begin{equation}\label{bound}
	\left\{
	\begin{aligned}
	&\int_X\gamma(\delta\psi_0)e^{-\psi_0}\sum_{j=1}^{N_{r+1}}|u_j^{(q,r+1)}|^2_{\omega,h_{q,r}\tilde{h}}dV_{X,\omega}\leq C_2 \quad &\text{if} \quad r<k-1,\\
	&\int_X\gamma(\delta\psi_0)e^{-\psi_0}\sum_{j=1}^{N_0}|u_j^{(q+1,0)}|^2_{\omega,h_{q,r}\tilde{h}}dV_{X,\omega}\leq C_2 \quad \quad &\text{if} \quad r=k-1.
	\end{aligned}
	\right.
	\end{equation}
	for some uniform constant $C_2$.  From \eqref{norm}, \eqref{bound} is equivalent to
	\begin{equation}\label{bound2}
	\left\{
	\begin{aligned}
	&\int_X\gamma(\delta\psi_0)e^{-\psi_0}\frac{\sum_{i=1}^{N_{r+1}}|u_i^{(q,r+1)}|^2_{\hat{h}^{qk+r+1}h_A}}{\sum_{i=1}^{N_r}|u_i^{(q,r)}|_{\hat{h}^{qk+r}h_A}^2}dV_{X,\omega}\leq C_2 \quad &\text{if} \quad r<k-1,\\
	&\int_X\gamma(\delta\psi_0)e^{-\psi_0}\frac{\sum_{i=1}^{N_0}|u_i^{(q,r+1)}|^2_{\hat{h}^{qk+k}h_A}}{\sum_{i=1}^{N_r}|u_i^{(q,r)}|_{\hat{h}^{qk+k-1}h_A}^2}dV_{X,\omega}\leq C_2 &\text{if} \quad r=k-1.
	\end{aligned}
	\right.
	\end{equation}
	Let us denote by 
	$$
	a_{qk+r}(x):=\sum_{i=1}^{N_r}|u_i^{(q,r)}|_{\hat{h}^{qk+r}h_A},
	$$
	which is a quasi-psh and bounded non-negative smooth function on $X$. By the integrability of $\log \gamma(\delta\psi_0)$ and $\psi_0$ with respect to the standard Lebesgue measure over $X$, combined with the concavity property of the logarithmic function as well as the Jensen inequality, we can find some constant $C_3$ and $C_4$ such that
	\begin{eqnarray}
	\int_X \log \frac{a_l}{a_{l-1}} dV_{X,\omega}\leq C_3-\int_X \log \gamma(\delta\psi_0)dV_{X,\omega}+\int_X \psi_0dV_{X,\omega}\leq C_4
	\end{eqnarray}
	for any $l\geq 1$. Since $a_1(x)$ is a bounded smooth function on $X$, we can also find a constant $C_5\geq C_4$ such that
	$$
	\int_X \log a_1dV_{X,\omega}\leq C_5.
	$$
	Combined these inequalities together we obtain
	$$
	\int_X \frac{\log a_l}{l}dV_{X,\omega}\leq C_5
	$$
	for any $l\geq 1$. Set $f_q:=\frac{\log a_{qk}}{q}$, and we have the following properties:
	\begin{enumerate}[(a)]
		\item for any $q\geq 1$, we have\label{a}
		$$	\int_X f_qdV_{X,\omega}\leq C_5	;$$
		\item  the inequality\label{b}
		$$k\Theta_{\hat{h}}(K_X+\tilde{L})+\hess f_q\geq -\frac{1}{q}\Theta_{h_A}(A)$$
		holds true in the sense of currents on $X$;
		\item on $X_y$ the following equality is satisfied\label{c}
		$$
		f_{q \upharpoonright X_y}=\log |\sigma|^2_{\hat{h}^k}+a_0(x)_{\upharpoonright X_y}
		$$
		where $a_0(x)=\log \sum_{i=1}^{N_0}|u_i^{(0,0)}|_{h_A}$ is a smooth function on $X$.
	\end{enumerate}
	
	By the mean value inequality
	for the psh functions, as a consequence of the properties (\ref{a}) and (\ref{b}), one can show the existence of a uniform upper bound for the functions $f_q$ over $X$. Thus the sequence $f_q(z)$ must have some subsequence which converges in $L^1$ topology on $X$ to the potential $f_{\infty}$, in the form of the regularized limit
	$$ 
	f_\infty(z):=\limsup_{\zeta\rightarrow z}\lim\limits_{q_\nu\rightarrow +\infty}f_{q_\nu}(\zeta),
	$$
	which satisfies
	$$k\Theta_{\hat{h}}(K_X+\tilde{L})+\hess f_{\infty}\geq 0$$
	as a current on $X$. Moreover, by Property (\ref{c}) $f_{\infty}$ is not identically $-\infty$ on $X_y$, as well as
	\begin{eqnarray}\label{uniform}
	f_{\infty}\geq \log  |\sigma|^2_{\hat{h}^k}+\oc(1)
	\end{eqnarray}
	pointwise on $X_y$. 
	
	Now we construct a singular hermitian metric $h_{\infty}$ on $(k-1)(K_X+L)$ defined by
	$$
	h_{\infty}:=\hat{h}^{k-1}e^{-\frac{k-1}{k}f_{\infty}}.
	$$
	Then $\Theta_{h_{\infty}}\big((k-1)(K_X+\tilde{L})\big)\geq 0$. Write $k(K_X+L)=K_X+(k-1)(K_X+\tilde{L})+\tilde{L}$, where $(k-1)(K_X+\tilde{L})$ is equipped with the singular hermitian metric $h_{\infty}$. Since
	$$
	|\sigma|^2_{\omega,\tilde{h} h_\infty }=|\sigma|^2_{\hat{h} h_\infty }=|\sigma|^{\frac{2(k-1)}{k}}_{h_\infty }\cdot |\sigma|^{\frac{2}{k}}_{\hat{h}}
	$$
	which is $\mathscr{C}^0$ bounded, we then can apply Proposition \ref{extension} to extend $\sigma$ to a global section $\Sigma\in H^0(X,k(K_X+\tilde{L}))$.

	In conclusion, for any regular value $y$ of the morphism $f$, the following morphism 
	$$
	H^0(X, K_X^{\otimes k}\otimes f^*L^{\otimes l})\rightarrow H^0\big(X_y,(K_X^{\otimes k}\otimes f^*L^{\otimes l})_{\upharpoonright X_y}\big)
	$$ 
	is always surjective for any $l>\frac{n}{\epsilon(L,y)}$. Thus Theorem \ref{mainot} is proved.
\end{proof}

\medskip

In order to improve the above quadratic bound to linear, we need to apply the twisted pluricanonical extension theorem in Section \ref{pluri} instead. First, we recall the  following result arising from birational geometry:
\begin{thm}\label{Demailly}
	Let $L$ be an ample line bundle over a projective $n$-fold $Y$, then the adjoint line bundle $K_Y+(n+1)L$ is semi-ample.
\end{thm} 
Based on the Mori theory, one observes that $n+1$ is the maximal length of extremal rays of smooth projective
$n$-folds, which shows that $K_Y+(n+1)L$ is nef. By the base-point-free theorem, one can even show that $K_Y+(n+1)L$ is semiample. In his work on the Fujita conjecture \cite{Dem96}, Demailly also gave an analytic proof for the fact that $K_Y+(n+1)L$ is nef.

\begin{proof}(Proof of Theorem \ref{sharp bound})
	Take a  a log resolution $\mu : X'\rightarrow  X$  of $(X,\Delta)$ such that
	$$
	K_{X'}=\mu^*(K_X+\Delta)+\sum_i{a_i}{E_i}-\sum_j{b_j}{F_j},
	$$
	where $a_i,b_j\in \mathbb{Q}_+$, and $\sum_{i,j} E_i+F_j$ is a divisor with simple normal crossing support.  By the assumption that $(X,\Delta)$ is klt and $\Delta$ is effective,  each ${E_i}$ is an exceptional divisor and  
	$0<b_j<1$ for each $b_j$. Denote $f':=f\circ \mu$. Then $f':X'\rightarrow Y$ is a surjective morphism between smooth projective manifolds.
	
	 It follows from Theorem \ref{Demailly} that $K_Y+(n+1)L$ can be equipped with a smooth hermitian metric $h_1$ with semi-positive curvature. If we further assume that $K_Y$ is pseudo-effective, $(m-1)K_Y+L$ is big for any $m\geq 1$ and thus can be equipped with a singular hermitian metric $h$ with mild singularities such that $\Theta_h\geq \epsilon\omega$ for some hermitian form $\omega$ over $Y$. Observe that if $m(K_X+\Delta)$ is a (integral) Cartier divisor, then $ma_i,mb_j\in \zbb^+$ for each $a_i,b_j$. Let us equip the divisor $\sum_j{mb_j}{F_j}$ with the canonical singular hermitian metric $h_2$ so that $\sqrt{-1}\Theta_{h_2}=\sum_j{mb_j}[{F_j}]$. For any point $y\in Y$, denote by $X'_y$ the fiber of $f':X'\rightarrow Y$. Let us denote  $P:=(m-1)f^* \big(K_Y+(n+1)L\big)+\sum_j{mb_j}{F_j}$ equipped with the (semi-positive curved) singular hermitian metric $h_P:=f^*h_1^{m-1}\cdot h_2$ in the general setting; and when $K_Y$ is pseudo-effective, denote $P:=f'^*\big((m-1)K_Y+L\big)+\sum_j{mb_j}{F_j}$ with the (semi-positive curved) singular hermitian metric $h_P:=f^*h\cdot h_2$. Recall that $\sum_j F_j$ is simple normal crossing and $0<b_j<1$. Then for general $y\in Y$, $\jc(h_{P\upharpoonright X'_y}^{\frac{1}{m}})=\oc_{X'_{y}}$.  By Theorem \ref{EKL}, when $k\geq n^2+1$, the Seshadri constant 
	 $$
	 \epsilon\big(kL, y\big)> n
	 $$
	 for a general $y\in Y$. So we can apply Theorem \ref{extension2} to show that,  the restriction map
\begin{eqnarray*}
H^0 \big(X', m K_{X'/Y} +P+f'^* (K_Y+lL) \big) \rightarrow H^0 \big(X'_y ,  (m K_{X'}+\sum_j{mb_j}{F_j})_{\upharpoonright X_y}\big)
\end{eqnarray*}
	is surjective for a generic $y$ in $Y$. In other word, 
\[
	H^0 \big(X', \mu^*(mK_X+m\Delta+lf^*L)+\sum_i{ma_i}{E_i}\big) \rightarrow	H^0\Big(X'_y ,  m\big(\mu^*(K_X+\Delta)+\sum_i{a_i}{E_i}\big)_{\upharpoonright X'_y}\Big)
		\]
		is surjective for a generic $y$ in $Y$ for
		any $l\geq m(n+1)+n^2-n$ in the general cases, and for $l\geq n^2+2$ when $K_Y$ is pseudo-effective.
		 
		Since each ${E_i}$ is an exceptional divisor of the birational morphism $\mu : X'\rightarrow  X$,   the natural inclusion
		$$
		H^0 \big(X', \mu^*(mK_X+m\Delta+lf^*L)\big)\rightarrow	H^0 \big(X', \mu^*(mK_X+m\Delta+lf^*L)+\sum_i{ma_i}{E_i}\big)
		$$
		is thus an isomorphism. Thus one also has the surjectivity
		\[
		H^0 \big(X', \mu^*(mK_X+m\Delta+lf^*L)\big) \rightarrow	H^0\Big(X'_y ,  m\big(\mu^*(K_X+\Delta)\big)_{\upharpoonright X'_y}\Big).
		\]
		In other words, the direct image
	$$f'_*\big(\mu^*(mK_X+m\Delta+lf^*L)\big)=f_*(mK_X+m\Delta)\otimes L^l$$
	is generated  by  global  sections  at  the generic points of $Y$ when $l\geq m(n+1)+n^2-n$,  and for $l\geq n^2+2$ when $K_Y$ is pseudo-effective.
	This completes the proof of Theorem \ref{sharp bound}.
\end{proof}

\begin{rem}
In a recent very interesting preprint \cite{LPS17}, Lombardi, Popa and Schnell proved that, when $Y$ is an Abelian variety, the sheaves $ 
f_*(mK_X)$ become globally generated after pullback by an isogeny. Their result is surprising since the extension does not require any local positivity. Since they use deep tools like GV sheaves, it is tempting to ask whether one can use analytic methods to give another proof of their results.
\end{rem}

\section{\sc On a Question of Demailly-Peternell-Schneider}\label{question dps}
In this section, we prove Theorem \ref{DPS01} and thus give an affirmative answer to  Problem \ref{DPS} in the case that both $X$ and $Y$ are smooth manifolds.

\begin{proof}[Proof of Theorem \ref{DPS01}]
Take a sufficiently divisible $q\in \mathbb{N}$ such that both $q\Delta$ and $qD$ are Cartier divisors, and a sufficient ample line bundle $A_X$ on $X$ such that $A_X+qD$ is ample a smooth hermitian metric $h$ on $A_X+qD$ such that $\ir\Theta_h\geq 3\omega$ for some K\"ahler metric $\omega$, and the direct image $f_*(A_X)$ is a torsion free coherent sheaf which is not only locally free but also globally generated over the Zariski open set $X^{\circ}:=X\setminus f^{-1}(Z)$. Then $f_*(A_X)$ is   locally free outside a  subvariety $W\subset Z$ of codimension at least 2. Set $r$ to be the generic rank of $f_*(A_X)$, and denote by
	$$\det f_*(A_X):=\Lambda^r  f_*(A_X)^{\star\star}$$
	to be the bidual of  $\Lambda^r  f_*(A_X)$ which is an invertible sheaf over $Y$, then there is coherent ideal sheaf $\ic$ supported on $W$ such that 
	$$
	\Lambda^r  f_*(A_X)=\det f_*(A_X)\otimes \ic.
	$$
Let us also choose a very ample line bundle $A_Y$   on $Y$ such that $A_Y-K_Y$ generates $n+1$ jets everywhere and $A_Y+\det f_*(A_X)$ is also an ample line bundle on $Y$. In particular, the Seshadri constant $\epsilon(A_Y-K_Y,y)>n$ for any $y$.
	\medskip 
	
It follows from   Definition \ref{restrict} that, for any $\mathbb{R}$-divisor $E$,	the restricted base locus of $E$ is defined by
	$$\mathbf{B}_-(E)=\bigcup_{m>0}\mathbf{B}(E+\frac{1}{m}A)$$
	where $A$ is an ample divisor, and the definition being independent of $A$. 
	Equivalently, in \cite{BDPP13}, it is shown that
	$$\mathbf{B}_-(E)=\bigcup_{m\in \mathbb{N}}\bigcap_{T}E_+(T),$$ 
	where $T$  runs over the set $c_1(E)[-\frac{1}{m}\omega]$ of all
	closed real $(1,1)$-currents $T\in c_1(E)$  such that $T\geq -\frac{1}{m}\omega$, and $E_+(T)$ denotes the locus where the Lelong numbers of $T$  are strictly
	positive. By \cite{Bou02}, there is always a current $T_{{\rm min},m}$
	which achieves minimum singularities and minimum Lelong numbers among all members of
	$c_1(E)[-\frac{1}{m}\omega]$, hence 
	$$
	\mathbf{B}_-(E)=\bigcup_{m\in \mathbb{N}}E_+(T_{{\rm min},m}).
	$$
	By Demailly's regularization theorem in \cite{Dem92b}, for every $m\in \mathbb{N}$, we can find a closed (1,1)-current $T_m\in c_1(E)$ with neat singularities such that
	$T_m\geq -\frac{2}{m}\omega$, and 
	$$ E_+(T_{{\rm min},2m})\subset E_+(T_{m})\subset E_+(T_{{\rm min},m}).$$ 
	Therefore, when $E$ is a Cartier divsor, there exists a singular hermitian metric $\tilde{h}_m$ on $E$ with neat singularities, such that the curvature current
	$$
	\ir \Theta_{\tilde{h}_m}= T_m\geq -\frac{2}{m}\omega.
	$$
	Set $E:=-q(K_X+D)-f^*(q\Delta)$. Since  $\mathbf{B}_-\big(-q(K_X+D)-f^*(q\Delta)\big)=\mathbf{B}_-\big(-(K_X+D)-f^*\Delta\big)$ does not project onto $Y$, thus for any $m\in \mathbb{N}$, $Z_m:=f\big(E_+(T_{m})\big)$ is a proper subvariety of $Y$, and the singular hermitian metric $\tilde{h}_m^{\otimes m}h$ on $-mq(K_X+D)-mqf^*\Delta+A_X+qD$ is smooth on $X\setminus f^{-1}(Z_m)$.
	
	\medskip
	
	For the $\mathbb{Q}$-effective divisor $D=\sum_{i=1}^{t}a_iD_i$, there is a canonical singular hermitian metric $h_D$ defined on $qD$, with the local weight $$\varphi_D=\sum_{i=1}^{t} qa_i\log |g_i|,$$
	where $g_i\in \Gamma(U,\oc_U)$ is a holomorphic function locally defining $D_i$ on some open set $U\subset X$. Therefore, the curvature current
	$$
	\ir \Theta_{h_D}=[qD]\geq 0,
	$$
	and thus $h_D$ is a singular hermitian metric with neat singularities. 
	
	Recall that $Z_D$ is denoted to be the  minimal  set containing $Z$, such that for every $y\notin Z_D$, the pair $(X_y, D_{\upharpoonright X_y})$ is also lc.
	Here we denote by $X_y:=f^{-1}(y)$. Since   $(X,D)$ is lc, thus $Z_D$ is  an at most countable union of proper subvarieties of $Y$. Indeed, the set
	$$
	Y_m:=\{y\notin Z | (X_y, (1-\frac{1}{m})D_{\upharpoonright X_y}) \mbox{ is klt}  \}
	$$
	is an Zariski open set of $Y$. Therefore, one has
	$$
	Z_D=\bigcup_{m=1}^{\infty}Y\setminus Y_m.
	$$
	Thus for the singular hermitian metric $h_m:=\tilde{h}_m^{\otimes m}hh_D^{\otimes m-1}$ on $-qmK_X+A_X-mqf^*\Delta$,  the multiplier ideal sheaf
\begin{eqnarray}\label{trivial multiplier}
\jc(h_{m\upharpoonright X_y}^{\frac{1}{qm}})=\jc((1-\frac{1}{m})D_{\upharpoonright X_y})=\oc_{X_y}
\end{eqnarray}
	for any $y\in Y_m\setminus Z_m$. Moreover, the curvature current $\ir\Theta_{h_m}\geq \omega$. 
	\medskip
	
	 Denote $L:=-mqK_X+A_X-mf^*(q\Delta)$. For any $y\in Y_m\setminus Z_m$, all the sections of the bundle $(mqK_X+L)_{\upharpoonright X_y}=\big(A_X-mf^*(q\Delta)\big)_{\upharpoonright X_y}$ extend
	locally near $y$,  and thus  it satisfies  Condition \eqref{con1}, \eqref{con2} and \eqref{con3} in Theorem \ref{extension2}. It then follows from \eqref{trivial multiplier} and Theorem \ref{extension2} that the restriction
	$$H^0 (X, mq K_{X/Y}  -mqK_X+A_X -mqf^*\Delta+f^* A_Y) \rightarrow H^0 (X_y ,  A_{X\upharpoonright X_y})$$
	 is surjective for any $y\in Y_m\setminus Z_m$.
	In other words, the direct image sheaf
	\begin{eqnarray}\label{direct image}
	f_*\big(mq K_{X/Y}  -mqK_X+A_X -mqf^*\Delta+f^* A_Y\big)=	(-K_Y-\Delta)^{\otimes mq}\otimes A_Y\otimes f_*(A_X)
	\end{eqnarray}
	is generated by global sections over $Y_m\setminus Z_m$, and by the assumption that $f_*(A_X)$ is locally free over $Y\setminus Z$, we conclude that the top exterior power 
	$$\Lambda^r \big((-K_Y-\Delta)^{\otimes mq}\otimes A_Y\otimes f_*(A_X)\big)=(-K_Y-\Delta)^{\otimes rmq}\otimes A_Y^{\otimes r}\otimes \det f_*(A_X)\otimes \ic$$
	is also generated by global sections over $Y_m\setminus Z_m$. In particular, for every $m\in \mathbb{N}$, the base locus 
	\begin{eqnarray}\label{base}
	{\rm Bs}\big((-K_Y-\Delta)^{\otimes rmq}\otimes A_Y^{\otimes r}\otimes \det f_*(A_X)\big)\subset Z_m\bigcup Y\setminus Y_m.
	\end{eqnarray}
	By our choice of $A_Y$, $rA_Y+\det f_*(A_X)$ is an ample line bundle on $Y$,  thus let  $m$ tends to infinity, we obtain the pseudo-effectivity of $-K_Y-\Delta$.  Moreover, it follows from \eqref{base}  that the restricted base locus $$\mathbf{B}_-(-K_Y-\Delta)\subset \bigcup_{m=1}^{\infty}Z_m\bigcup Y\setminus Y_m=f\big(\mathbf{B}_-(-K_X-D-f^*\Delta)\big)\bigcup Z_D.$$ 
	Hence Claim \eqref{claim2}  is proved.
	
	\medskip
	
Let us prove Claim \eqref{claim1}.	Let $p : Y'\rightarrow Y $ be a log-resolution of singularities of $Y$.
Let $\mu : X'\rightarrow  X$ be a log resolution of $(X,D)$, such that the induced	rational map $f':X'\rightarrow Y'$
is in fact a morphism. We have
the following commutative diagram:
\begin{displaymath}
\xymatrix{ X' \ar[d]^{f'} \ar[dr]^{g} \ar[r]^{\mu} & X \ar[d]^-{f}\\
	Y' \ar[r]^{p} & Y.
}
\end{displaymath}	
Moreover,   $K_{X'}+D'=\mu^*(K_X+D)+F$, where 
 $D'$ and $F$ are both effective 	$\mathbb{Q}$-divisors without common components. Moreover, $D'$ is log canonical and $F$ is exceptional. For any $q\in \mathbb{N}$ such that $q(K_X+D)$ is a Cartier divisor, both $qF$ and $qD'$ are also Cartier.
	By \cite[Lemma 2.6]{BBP13}, we also have
	$$
	\mu\Big({\rm NNef}\big(-\pi^*(K_X+D)\big)\Big)\subset {\rm NNef}\big(-(K_X+D)\big).
	$$
	Thus by the assumption of the theorem,  we have
	\begin{eqnarray}
	f'\big({\rm NNef}(-K_{X'}-D'+F)\big)\subsetneq Y'.
	\end{eqnarray} 
	Repeat the proof of Claim \eqref{claim2} with $K_X+D$ replaced by $K_X'+D'-F$, one can prove that, after one fixes certain ample divisors $A_{X'}$ and $A_{Y'}$ over $X'$ and $Y'$, for any $m\in \mathbb{N}$,
	the restriction
	$$H^0 (X', mq K_{X'/Y'}  -mqK_{X'}+mqF+A_{X'}+f'^* A_{Y'}) \rightarrow H^0 (X'_y ,  (mqF+A_{X'})_{\upharpoonright X'_y})$$
	is surjective for a general point $y$ in $Y$. Since $qF$ is an effective exceptional divisor, the natural inclusion
	$$
	H^0 (X', mq K_{X'/Y'}  -mqK_{X'}+A_{X'}+f'^* A_{Y'})\rightarrow H^0 (X', mq K_{X'/Y'}  -mqK_{X'}+mqF+A_{X'}+f'^* A_{Y'})
	$$
	is an isomorphism, and thus
	the restriction
	$$H^0 (X', mq K_{X'/Y'}  -mqK_{X'}+A_{X'}+f'^* A_{Y'}) \rightarrow H^0 (X'_y ,  A_{X \upharpoonright X'_y})$$
	is surjective for a general point $y$ in $Y'$. By the same proof as above, we conclude that $-K_{Y'}$ is pseudo-effective, and it follows from Lemma \ref{inherience} below that $-K_Y$ is also pseudoeffective.   We finish the proof of Claim \eqref{claim1}. 
\end{proof}

\begin{lem}\label{inherience}
	Let $\mu:Y'\rightarrow Y$ be a birational morphism from the smooth projective variety $Y'$ to the normal $\mathbb{Q}$-Gorenstein variety $Y$. When $-K_{Y'}$ (resp. $K_{Y'}$) is pseudo-effective, so is $-K_Y$ (resp. $K_Y$).
\end{lem}
\begin{proof}
	For any sufficiently large $m\in \mathbb{N}$ such that $mK_{Y}$ is Cartier, there exists effective exceptional divisors $E$ and $F$ on $Y'$ such that 
	$$
	\mu^*(-mK_Y)=-mK_{Y'}+E -F.
	$$ 
	Take an ample divisor $A$ over $Y$ such that, for some effective exceptional divisors $G$, $\mu^*A-G$ is also ample over $Y'$. When $-K_{Y'}$  is pseudo-effective,  $-mK_{Y'}+\mu^*(A)-G$ is big, and thus for a sufficiently large $l\in \mathbb{N}$, there exists a non-zero section 
	$$
	s\in H^0\big(Y',l\mu^*(A-mK_Y)-lG+lF-lE\big).
	$$
	The non-zero section $s\cdot lG\cdot lE\in H^0\big(Y',l\mu^*(A-mK_Y)+lF\big)$ gives rise to a section
	$
s'\in	H^0(Y,lA-lmK_Y).
	$
	Since $m$ can be chosen arbitrarily large, we conclude that $-K_Y$ is pseudoeffective. The same proof also holds when $K_{Y'}$ is pseudoeffective. 
\end{proof}

If  $f$ is a smooth fibration, ${\rm Supp}(D)$ is a simple normal crossing divisor, and ${\rm Supp}(D)$ is relatively normal	crossing over $Y$, then the condition that $(X,D)$ is lc implies that $(X_y,D_{\upharpoonright X_y})$ is also lc for every $y\in Y$. Thus $Z_D=\emptyset$.  If $-K_X-D-f^*\Delta$ is nef, then  $\mathbf{B}_-(-K_X-D-f^*\Delta)=\emptyset$.  Thus from Theorem \ref{DPS01},  $\mathbf{B}_-(-K_Y-\Delta)$ is also empty which implies that  $-K_Y-\Delta$ is  nef. This completes our proof of Theorem \ref{FG2}.

By setting $D=0$ and $\Delta=0$ in Theorem \ref{FG2}, the following theorem by Miyaoka is a direct consequence.
\begin{thm}[Miyaoka]\label{Miyaoka}
	Let $f:X\rightarrow Y$
	be a smooth morphism between smooth projective manifolds $X$ and $Y$. If $-K_X$ is nef, then so is $-K_Y$.
\end{thm}

\begin{rem}
	The original proof of Miyaoka \cite{Miy93} relies on the mod $p$ reduction arguments. There is also another  Hodge theoretic proof by Fujino and Gongyo without using the mod $p$ reduction arguments \cite{FG14}.
\end{rem}

\begin{rem}
	In \cite{CZ13}, M. Chen and Q. Zhang proved the similar result  as Claim \eqref{claim1} in Theorem \ref{DPS01}, under the stronger assumption that    $-(K_X+D)$ is nef. In a very recent preprint \cite{Ou17}, W. Ou extended the theorem by Chen-Zhang to the rational dominant maps, which was a crucial step in his proof of the \emph{generic nefness conjecture} for tangent sheaves by T. Peternell \cite[Conjecture 1.5]{Pet12}.
\end{rem}

\section{\sc On the Inheritance of the Image}

\subsection{On the Images of Weak KLT Fano Manifolds}
One says that a projective manifold $X$ is   weak Fano   if $-K_X$ is big and nef. In the series of articles \cite{FG12} and \cite{FG14}, Fujino and Gongyo studied the image of weak Fano manifolds. They proved the following theorem:
\begin{thm}[Fujino-Gongyo]\label{FG}
	Let $f:X\rightarrow Y$ be a smooth fibration between two smooth manifolds $X$ and $Y$. If $X$ is weak Fano, then so is $Y$.
\end{thm}
In this section, we will prove Claim \eqref{claima} in Theorem \ref{main2}:
\begin{proof}[Proof of Claim \eqref{claima} in Theorem \ref{main2}]
Take   a very ample line bundle $A_Y$ over $Y$ such that $A_Y$ generates $n+1$ jets everywhere. Since $-K_X-D-f^*\Delta$ is big, we can find a a sufficiently divisible $a\in \mathbb{N}$ such that $-a(K_X+D+f^*\Delta)-2f^*A_Y$ is an effective line bundle. Fix any effective divisor $E\in |-a(K_X+D+f^*\Delta)-2f^*A_Y|$. Since $(X,D)$ is klt,  then there exists a sufficiently divisible integer $m>a$ such that   the multiplier ideal sheaves 
	\begin{eqnarray}\label{multiplier 1}
	\jc(\frac{1}{m-1}E_{\upharpoonright X_y})=\jc(\frac{m}{m-1}D_{\upharpoonright X_y})=\oc_{X_y}
	\end{eqnarray} for the generic fiber $X_y$, and both $mD$ and $m\Delta$ are  $\mathbb{Z}$-divisors. We can also find a singular hermitian metric $h_1$ with neat singularities on $-(m^2-a)(K_X+D+f^*\Delta)$ such that $\ir \Theta_{h_1}\geq \tilde{\omega}$ for some K\"ahler metric $\tilde{\omega}$ on $X$. Take some small rational number $\epsilon>0$ such that  $\jc(h_{1\upharpoonright X_y}^{\epsilon})=\oc_{X_y}$ for the generic fiber $X_y$.
	
	On the other hand, since the non-nef locus ${\rm NNef}(-K_X-D-f^*\Delta)$ does not project onto $Y$, it follows from the proof of Theorem \ref{DPS01} in Section \ref{question dps} that, one can  find a singular hermitian metric $h_\epsilon$ over $-(m^2-a)(K_X+D+f^*\Delta)$ with neat singularities, such that $\ir \Theta_{h_\epsilon}\geq -\epsilon\tilde{\omega}$ and the singularities of $h_\epsilon$ does not project onto $Y$. Set $h:=h_1^\epsilon h_\epsilon^{1-\epsilon}$ which is also a hermitian metric on $-(m^2-a)(K_X+D+f^*\Delta)$, then we have $\ir \Theta_{h}\geq \epsilon^2\tilde{\omega}$  and the multiplier ideal sheaf
	\begin{eqnarray}\label{multiplier 2}
	\jc(h_{\upharpoonright X_y})=\oc_{X_y}
	\end{eqnarray} 
	for the generic fiber $X_y$. 	
	
	Take a generic fiber $X_y$ of $f$ such that $y$ is the regular value of $f$, and both (\ref{multiplier 1}) and (\ref{multiplier 2}) are satisfied. We equip the line bundle $-m^2(K_X+D+f^*\Delta)-2f^*A_Y+m^2 D$ with the singular hermitian metric $h_0:=h_Ehh_D^{\otimes m^2}$, where $h_E$ (resp. $h_D$) is the tautological singular hermitian metric on $-a(K_X+D+f^*\Delta)-2f^*A_Y$ ( resp. $D$) induced by the effective divisor $E$ (resp. $D$), such that
	$$
	\ir \Theta_{h_E}=[E] (\mbox{ resp. } \ir \Theta_{h_D}=[D]).
	$$ 
	Then we claim that the multiplier ideal sheaf $\jc(h_{0 \upharpoonright X_y}^{\frac{1}{m^2}})=\oc_{X_y}$. Indeed, for any $s\in \oc_{X_y,z}$, let $\varphi_E$, $\varphi_D$ and $\varphi$   be the weights of the metric $h_E$, $h_D$ and $h$ on a small neighborhood $U\subset X_y$ of a point $z\in X_y$. Then by the H\"older inequality we have
	$$
	\int_{U}|s|^2e^{-\frac{\varphi_E+\varphi}{m^2}+\varphi_D}\leq (\int_U |s|^2e^{-\varphi})^{\frac{1}{m^2}}\cdot (\int_U |s|^2e^{-\frac{\varphi_E}{m-1}})^{\frac{m-1}{m^2}}\cdot (\int_U |s|^2e^{-\frac{m}{m-1}\varphi_D})^{\frac{m-1}{m}}<+\infty.
	$$
	Here we use the conditions (\ref{multiplier 1}) and (\ref{multiplier 2}). By  applying Theorem \ref{extension2} with $L=-m^2(K_X+f^*\Delta)-2f^*A_Y$ endowed with the singular hermitian metric $h_0$, for general $y\in Y$, we obtain the desired surjectivity:
	$$H^0 \big(X, m^2 K_{X/Y} + (-m^2K_X-m^2f^*\Delta-2f^*A_Y) +f^* A_Y\big) \rightarrow H^0 \big(X_y ,  f^*(-m^2K_Y-m^2\Delta-A_Y)_{\upharpoonright X_y}\big)=\cb^{ l},$$
	where $l$ is the number of the connected components of $X_y$.
	In particular, we have the non-vanishing 
	$$
	H^0\big(X, f^*(-m^2K_Y-m^2\Delta-A_Y)\big)\neq 0.
	$$
	Now we claim that  $-m^2K_Y-m^2\Delta-A_Y$ is a pseudo-effective line bundle over $Y$. Indeed, we first take a Stein factorization of $f$
	$$
	X\xrightarrow{f'}Y'\xrightarrow{p}Y,
	$$
	where $p:Y'\to Y$ is a finite surjective morphism and the morphism $f':X\rightarrow Y'$ has connected fibers. Then we have an isomorphism
	$$
	f'_*:H^0\big(X, f^*(-m^2K_Y-m^2\Delta-A_Y)\big)\xrightarrow{\cong} H^0\big(Y',p^*(-m^2K_Y-m^2\Delta-A_Y)\big),
	$$
	which implies that the line bundle $p^*(-m^2K_Y-m^2\Delta-A_Y)$ is effective. Since $p:Y'\to Y$ is a finite surjective morphism, by a result of S. Boucksom \cite[Proposition 4.2]{Bou02}, $-m^2K_Y-m^2\Delta-A_Y$ is   a pseudo-effective line bundle,  which also shows that $-K_Y-\Delta$ is big.
\end{proof}

Therefore, we can extend Theorem \ref{FG} to the weak klt Fano cases:
\begin{proof}[Proof of Theorem \ref{main3}]
	Since $f$ is a smooth fibration, $(X,D)$ is klt, and $(X_y,D_{\upharpoonright X_y})$ is also klt for every $y\in Y$, from the very definition of $Z_D$ in Theorem \ref{DPS01} we see that  $Z_D=\emptyset$. By the nefness of $-(K_X+D)-f^*\Delta$, the set  $$\mathbf{B}_-\big(-(K_X+D)-f^*\Delta\big)=\emptyset.$$
	Thus from Theorem \ref{DPS01} we conclude that $-K_Y-\Delta$ is nef. The bigness of $-K_Y-\Delta$ follows from Theorem \ref{main3} directly. This completes the proof.
\end{proof}

By setting $D=0$ and $\Delta=0$ in Theorem \ref{main3},  we obtain Theorem \ref{FG} directly.

\begin{rem}
	If we only  assume that $-K_X$ is big, then the following example given in \cite{FG12} shows that, even if $f$ is smooth, $-K_Y$ is not big.
\end{rem}
\begin{example}
	Let $E\subset \pb^2$ be a smooth cubic curve. Consider
	$f:X=\pb_E(\oc_E\oplus \oc_E(1))\rightarrow E=Y$. Then, we see that $-K_X$ is big.	However, $-K_Y$ is not big since $E$ is a smooth elliptic curve.
\end{example}

\medskip

It is noticeable that, in \cite{Pau12} S. Boucksom pointed out that the following
theorem, which is a special case of Theorem 1.2 in \cite{Ber09}, implies    \cite[Theorem 4.1]{FG12} or  \cite[Corollary 2.9]{KMM92}:
\begin{thm}[Boucksom-P\u{a}un]
	Let $f:X\rightarrow Y$ be a smooth fibration between two smooth manifolds. If $-K_X$ is semi-positive (strictly positive), then $-K_Y$ is semi-positive (strictly positive).
\end{thm}

\medskip

Finally, let us mention that, in \cite{FG12}, the authors raised the following conjecture, which was solved very recently by C. Birkar and Y. Chen \cite{BC16}:
\begin{thm}[Fujino-Gongyo, Birkar-Chen]
	Let $f:X\rightarrow Y$ be a smooth fibration between two smooth projective manifolds. If $-K_X$ is semi-ample, then so is $-K_Y$.
\end{thm}
The proof in \cite{BC16} relies on very deep consequences of the minimal model program in birational geometry and of Hodge theory. It is an interesting question to know  whether we can use pure analytic methods to give a new proof of this theorem.

\subsection{On the Rational Connectedness of the Image}
By Mori's bend-and-break, Fano varieties are uniruled; in fact by \cite{Cam92,KMM92} a stronger result
holds:
the projective Fano variety is rationally connected.
Later on Q. Zhang and Hacon-McKernan   proved that the same conclusion holds for
a klt pair $(X,D)$ such that $-(K_X+D)$ is big and nef \cite{Zha06,HM07}. This was generalized by Broustet and Pacienza \cite[Theorem 1.2]{BrP11}, who proved
that a klt pair $(X,D)$ with $-(K_X+D)$   big is rationally connected modulo the
non-nef locus of $-(K_X+D)$, that is, there
exists an irreducible 
component $V$ of ${\mathbf{B}_-}\big(-(K_X+D)\big)$ such that for any general point $x$ of $X$ there exists a rational 
curve $R_x$ passing through $X$ and intersecting
$V$. Moreover, they also proved the following result for the image:
\begin{thm}[Broustet-Pacienza]\label{BP}
	Let $(X,D)$ be a   pair such that $-(K_X+D)$ is  big. Let $f:X\dashrightarrow Y$ be
	a dominant rational map with 
	connected fibers such that the restriction of $f$  to  ${\rm NNef}\big(-(K_X+D)\big)\bigcup {\rm Nklt}(X,D)$ does not dominate $Y$, then $Y$ is uniruled.
\end{thm}

In this subsection, we will refine their results in a more general setting. First, we need to  prove Claim \eqref{claimb} in Theorem \ref{main2}:
\begin{proof}[Proof of Claim \eqref{claimb} in Theorem \ref{main2}]
	Let $p : Y'\rightarrow Y $ be a log-resolution of singularities of $Y$.
	Let $\pi : X'\rightarrow  X$ be a log resolution of $(X,D)$, such that the induced	rational map $f':X'\rightarrow Y'$
	is in fact a morphism. We have
	the following commutative diagram:
	\begin{displaymath}
	\xymatrix{ X' \ar[d]^{f'} \ar[dr]^{g} \ar[r]^{\pi} & X \ar[d]^-{f}\\
		Y' \ar[r]^{p} & Y.
	}
	\end{displaymath}
	Let $D'$ be an effective
	$\mathbb{Q}$-divisor on $X'$ such that $\pi_*(D')=D$ and	$K_{X'}+D'=\pi^*(K_X+D)+F$,
	with $F$ effective and not having common components with $D'$.  Note that
	$$
	\pi\big({\rm Nklt}(X',D') \big)\subset {\rm Nklt}(X,D).
	$$
	By \cite[Lemma 2.6]{BBP13}, we also have
	$$
	\pi\Big({\rm NNef}\big(-\pi^*(K_X+D)\big)\Big)\subset {\rm NNef}\big(-(K_X+D)\big).
	$$
It then follows from the assumption of the theorem that
	\begin{eqnarray}\label{notdomin}
	f'\big({\rm NNef}(-K_{X'}-D'+F)\bigcup {\rm Nklt}(X',D')\big)\subsetneq Y'
	\end{eqnarray} 
	Take a very ample line bundle $A_{Y'}$ over $Y'$ such that $A_{Y'}-K_{Y'}$ generates $n+1$ jets everywhere. We can take an ample line $A_{Y'}:=p^*A_{Y'}-E_Y$ over $Y'$, where $E_Y=\sum_j c_jE_j'$'s are exceptional divisors of $p$.  Since $-K_X-D$ is big, so is $-K_{X'}-D'+F$, and  we can find a sufficiently divisible $a$ such that $-a(K_{X'}+D'-F)-2f'^*A_{Y'}$ is an effective line bundle. Fix any effective divisor $E\in |-a(K_{X'}+D'-F)-2f'^*A_{Y'}|$. By \eqref{notdomin},  for any sufficiently divisible $m>a$,   the multiplier ideal sheaf 
	\begin{eqnarray}\label{multiplier 3}
	\jc(\frac{1}{m-1}E_{\upharpoonright X'_{y}})=\jc(\frac{m}{m-1}D'_{\upharpoonright X'_{y}})=\oc_{X'_{y}}
	\end{eqnarray} for the generic (smooth) fiber $X'_{y}$  of $f':X'\rightarrow Y'$. We can also find a singular hermitian metric $h_1$ with neat singularities on $-(m^2-a)(-K_{X'}-D'+F)$ such that $\ir \Theta_{h_1}\geq \tilde{\omega}$ for some K\"ahler metric $\tilde{\omega}$ on $X'$. Take some small rational number $\epsilon>0$ such that  $\jc(h_{1\upharpoonright X'_{y}}^{\epsilon})=\oc_{X'_{y}}$ for the generic fiber $X'_{y}$.
	
	On the other hand, it follows from \eqref{notdomin} that the non-nef locus ${\rm NNef}(-K_{X'}-D'+F)$ does not project onto $Y'$, and from the proof of Theorem \ref{DPS01} in Section \ref{question dps},    one can  find a singular hermitian metric $h_\epsilon$ over $-(m^2-a)(K_{X'}+D'-F)$ with neat singularities, such that $\ir \Theta_{h_\epsilon}\geq -\epsilon\tilde{\omega}$ and the singularities of $h_\epsilon$ does not project onto $Y$. Set $h:=h_1^\epsilon h_\epsilon^{1-\epsilon}$ which is also a hermitian metric on $-(m^2-a)(K_{X'}+D'-F)$, then we have $\ir \Theta_{h}\geq \epsilon^2\tilde{\omega}$  and the multiplier ideal sheaf
	\begin{eqnarray}\label{multiplier 4}
	\jc(h_{\upharpoonright X'_{y}})=\oc_{X'_{y}}
	\end{eqnarray} 
	for the generic fiber $X'_{y}$.	
	
	Take a generic regular value $y\in Y'$ of $f'$ such that the fiber $X'_{y}$  is reduced and smooth,  and both (\ref{multiplier 3}) and (\ref{multiplier 4}) are satisfied. We equip the line bundle $-m^2(K_{X'}+D'-F)-2f'^*A_{Y}+m^2 D'$ with the singular hermitian metric $h_0:=h_Ehh_{D'}$, where $h_E$ (resp. $h_{D'}$) is the tautological singular hermitian metric on $-a(K_{X'}+D'-F)-2f'^*A_{Y'}$ ( resp. $m^2D'$) induced by the effective divisor $E$ (resp. $m^2D'$), such that
$$
\ir \Theta_{h_E}=[E] (\mbox{ resp. } \ir \Theta_{h_{D'}}=m^2[D']).
$$ 
Then we claim that the multiplier ideal sheaf $\jc(h_{0 \upharpoonright X'_{y}}^{\frac{1}{m^2}})=\oc_{X'_{y}}$. Indeed, for any $s\in \oc_{X'_{y},z}$, let $\varphi_E$, $\varphi_{D'}$ and $\varphi$   be the weights of the metric $h_E$, $h_{D'}$ and $h$ on a small neighborhood $U\subset X'_{y}$ of a point $z\in X'_{y}$. Then by the H\"older inequality we have
$$
\int_{U}|s|^2e^{-\frac{\varphi_E+\varphi+\varphi_{D'}}{m^2}}\leq (\int_U |s|^2e^{-\varphi})^{\frac{1}{m^2}}\cdot (\int_U |s|^2e^{-\frac{\varphi_E}{m-1}})^{\frac{m-1}{m^2}}\cdot (\int_U |s|^2e^{-\frac{1}{m(m-1)}\varphi_{D'}})^{\frac{m-1}{m}}<+\infty.
$$	
	Here we use the conditions (\ref{multiplier 3}) and (\ref{multiplier 4}). By  applying Theorem \ref{extension2} to the surjective morphism $f':X'\rightarrow Y'$  with $L=-m^2(K_{X'}-F)-2f'^*A_{Y'}$ endowed with the singular hermitian metric $h_0$, we obtained the desired surjectivity:
	$$H^0 \big(X', m^2 K_{X'/Y'} + (-m^2K_{X'}-2f'^*A_{Y'}+m^2F) +f'^* A_{Y'}\big) \rightarrow H^0 \big(X'_{y} ,  (-m^2f'^*K_{Y'}-f'^*A_{Y'}+m^2F)_{\upharpoonright X'_{y}}\big)\neq 0$$
	for general fiber $X'_y$.
	In particular, we have the non-vanishing 
	$$
	H^0\big(X', f'^*(-m^2K_{Y'}-A_{Y'})+m^2F\big)\neq 0.
	$$
Since $X$ is normal with $F$ the exceptional divisors of the birational morphism $\pi:X'\rightarrow X$, the natural isomorphism
	$$
	H^0\big(X', f'^*(-m^2K_{Y'}-A_{Y'})\big)\xrightarrow{\cong} H^0\big(X', f'^*(-m^2K_{Y'}-A_{Y'})+m^2F\big) 
	$$
	is an isomorphism.
	 Take a Stein factorization of $f'$
	$$
	X'\xrightarrow{\tilde{f}}\tilde{Y}\xrightarrow{\mu}Y',
	$$
	where $\mu:\tilde{Y}\to Y'$ is a finite surjective morphism and the morphism $\tilde{f}:X'\rightarrow \tilde{Y}$ has connected fibers. Then we have an isomorphism
	$$
	\tilde{f}_*:H^0\big(X, f'^*(-m^2K_{Y'}-A_{Y'})\big)\xrightarrow{\cong} H^0\big(\tilde{Y},\mu^*(-m^2K_{Y'}-A_{Y'})\big),
	$$
	which implies that the line bundle $\mu^*(-m^2K_{Y'}-A_{Y'})$ is effective. Since $p:Y'\to Y$ is a finite surjective morphism, by a result of S. Boucksom \cite[Proposition 4.2]{Bou02}, $-m^2K_{Y'}-A_{Y'}$ is   a pseudo-effective line bundle. Take $l$ sufficiently large such that $lK_{Y'}$ is Cartier and there exists a section
	$$
	s\in H^0(Y',-lK_Y'-A_{Y'}).
	$$
	Since 
	$$
	-p^*(lK_Y)=-lK_{Y'}+E_1-E_2
	$$
	with both $E_1$ an $E_2$ effective and exceptional, by $A_{Y'}:=p^*A_Y-E_Y$ with $E_Y$ effective exceptional divisor,  one has
	$$
	s\cdot E_1 \in H^0\big(Y',	-p^*(lK_Y+A_Y)+E_2+E_Y\big),
	$$
which in turn gives rise to a non-zero section $s'\in H^0(Y,	-lK_Y-A_Y)$. Thus $-K_Y$ is big.
\end{proof}

\begin{proof}[Proof of Theorem \ref{RC}]	The proof is more or less direct. By  Claim \eqref{claimb} in Theorem \ref{main2} we see that $-K_Y$ is big. By Broustet-Pacienza's Theorem \cite[Theorem 1.2]{BrP11},  $Y$ is rationally connected modulo the non-nef locus ${\rm NNef}(-K_Y)$. The theorem is thus proved.
\end{proof}

\bigskip
\noindent
\textbf{Acknowledgements.}
The first version of this work was finished during my PhD studies. I would like to warmly thank my former  advisor Professor J.-P. Demailly for his numerous suggestions on this work, and also for his constant encouragements and supports. 
I thank Professors S\'ebastien Boucksom, Gianluca Pacienza and Mihnea Popa   for  the discussions and improvements they suggested. I also thank Junyan Cao for his interests and  kindly sharing his ideas to me,   Chen Jiang for the fruitful discussions and improvements he suggested, and Yajnaseni Dutta for the patient explanations of her work. The work is supported by ERC ALKAGE project and the postdoctoral project funded by Labex IRMIA.

\end{document}